\documentclass[11pt]{article} 

\usepackage{amsmath,amssymb,enumerate,theorem,hyperref}
\usepackage[T1]{fontenc}
\usepackage[latin1]{inputenc}
\def\r{\rightarrow}

\newcommand{\fdem}{\hspace*{\fill}~$\Box$\par\endtrivlist\unskip}

\newcommand{\E}{\mathbb{E}}     
\renewcommand{\P}{\mathbb{P}}     
\renewcommand{\L}{\mathbb{L}}

\newcommand{\N}{\mathbb{N}}     

\newcommand{\R}{\mathbb{R}}     
     
\newcommand{\C}{\mathbb{C}} 
 
\newcommand{\X}{\mathbb{X}} 
\newcommand{\T}{\mathbb{T}}

\renewcommand{\r}{\mathop{\rightarrow}}
\newcommand{\cA}{\mbox{$\cal A$}}

\newcommand{\cD}{\mbox{$\cal D$}}
\newcommand{\cE}{\mbox{$\cal E$}}

\newcommand{\cG}{\mbox{$\cal G$}}

\newcommand{\cM}{\mbox{$\cal M$}}
\newcommand{\cN}{\mbox{$\cal N$}}
\newcommand{\cP}{\mbox{$\cal P$}}

\newcommand{\cS}{\mbox{$\cal S$}}

\newcommand{\cY}{\mbox{$\cal Y$}}

\newtheorem{theo}{Theorem}[section]
\newtheorem{pro}{Proposition}[section]
\newenvironment{proof}[1]{\textit{Proof#1.\,}}{\fdem}
\newtheorem{lem}{Lemma}[section]
\newtheorem{rem}{Remark}[section]

\newtheorem{ass}{AC\hspace*{-3pt}}
\newtheorem{assu}{U\hspace*{-3pt}}

{\theoremstyle{break}\theorembodyfont{\rmfamily}
}

\newcommand{\IP}{\textbf{(I-A)}} 
\newcommand{\M}[1]{\boldsymbol{(\mathrm{M}#1)}} 
\newcommand{\NL}{\textbf{(NL)}}

\begin{document}

\title{Local limit theorem for densities of the additive component of a finite Markov Additive Process}  

\author{Loïc HERV\'E and James LEDOUX  \footnote{
Université
Européenne de Bretagne, I.R.M.A.R. (UMR-CNRS 6625), INSA de Rennes}}

\maketitle

\begin{abstract}
In this paper, we are concerned with centered Markov Additive Processes $\{(X_t,Y_t)\}_{t\in\T}$ where the driving Markov process $\{X_t\}_{t\in\T}$ has a finite state space. Under suitable conditions, we provide a local limit theorem for the density of the absolutely continuous part of the probability distribution of $t^{-1/2}Y_t$ given $X_0$. The rate of convergence and the moment condition are the expected ones with respect to the i.i.d case. An application to the joint distribution of local times of a finite jump process is sketched. \\[1mm]
\end{abstract}

\section{Introduction}

When $\{X_k\}_{k\geq1}$ is a sequence of centered independent and identically distributed (i.i.d.) real valued random variables such that $Y_n:=\sum_{i=1}^n X_i$ has a bounded density for some $n$, it is well-known that the density $f_n$ of  $n^{-1/2}Y_n$ satisfies the following  Local Limit Theorem (LLT)  
$$\lim_{n\r+\infty}\sup_{y \in \R} \big|f_{n}(y) - \eta(y)\big| = 0,$$
where $\eta(\cdot)$ is the density of the Gaussian distribution $\cN(0,\sigma^2)$, with $\sigma^2:=\E[X_1^2]$; see \cite{Gne54} and \cite{IbrLin71,Fel71} for detailed discussions. If $X_1$ has a bounded density and a third moment, then the rate of the previous convergence is $O(n^{-1/2})$; see \cite{SirSha65,Sir66,korZhu98}. 
 
This paper extends the last result to centered Markov Additive Processes (MAP) $\{(X_t,Y_t)\}_{t\in \T}$ with state space $\X \times \R^d$, where  $\X:=\{1,\ldots,N\}$ and $\T:=\N$ or $\T:=[0,\infty)$. Recall from \cite{Asm03} that $\{(X_t,Y_t)\}_{t\in \T}$ is a Markov process on $\X \times \R^d$ with  a transition semi-group, denoted by $\{Q_t\}_{t\in\T}$, which satisfies 
	\begin{equation} \label{Def_Add}
	 \forall (k,y)\in\X\times \R^d, \ \forall (\ell,B) \in \X\times B(\R^d),\quad 
	 Q_t(k,y;\{\ell\}\times B)=  Q_t(k,0;\{\ell\}\times B-y).
\end{equation}
The transition semi-group of the driving Markov process $\{X_t\}_{t\in\T}$ is denoted by $\{P_t\}_{t\in\T}$. The stochastic $N\times N$-matrix $P:=P_1$ is assumed to be irreducible and aperiodic. Moreover the mass of the singular part of the conditional probability distribution of $t^{-1/2}Y_t$ given $X_0=k$ is supposed to converge exponentially fast to zero. Let  $f_{k,t}(\cdot)$ be the density of the  absolutely continuous part of this conditional distribution. Under a third moment condition on $Y_t$ and some conditions on $f_{k,t}(\cdot)$ and its Fourier transform (see \textbf{(AC\ref{AS2})-(AC\ref{AS0})}), we prove in Theorem~\ref{main-id} that, for every $k\in\X$, the density $f_{k,t}(\cdot)$ 
satisfies essentially the following property: 
\begin{equation*} 
\sup_{y \in \R^d} \big|f_{k,t}(y) - \eta_{_\Sigma}(y)\big| = O\big( t^{-1/2}\big) 
\end{equation*}
where $\eta_{_\Sigma}(\cdot)$ denotes the density of $\cN(0,\Sigma)$. Matrix $\Sigma$ is the asymptotic covariance provided by the Central Limit Theorem (CLT) and is assumed to be invertible. Our moment condition and rate of convergence are the expected ones with respect to the i.i.d.~case. The proof of Theorem~\ref{main-id} is based on the spectral method (e.g.~see \cite{GuiHar88,HenHer01} when $\T:=\N$, and \cite{FerHerLed12} when $\T:=[0,+\infty)$).

To the best of our knowledge, Theorem~\ref{main-id}  is new. The known contribution to LLT for densities of additive components of MAP is in \cite{San08}, where only the discrete time case is considered and exponential-type moment condition on $Y_1$ is assumed (the rate of convergence is not addressed). Note that, for discrete time, our Theorem~\ref{main-id} only requires a third moment condition on $Y_1$. Moreover it is worth noticing that the probability distribution of $Y_t$ is not assumed to be absolutely continuous with respect to the Lebesgue measure on $\R^d$. An application to the joint distribution of local times of a finite jump process is provided in Section~\ref{appli-local-times}. The proof of Theorem~\ref{main-id} is given in the last section. In order to save space, some details are reported in a companion paper which is referred to as [HL]\footnote{Available at \url{http://arxiv.org/abs/1305.5644}}.

This work was motivated by a question in relation
with a self-attracting continuous time process called the Vertex
Reinforced Jump Process (VRJP) recently investigated by \cite{SabTar11}. This
process is closely related to the Edge Reinforced Random Walk (ERRW) and
was instrumental in the proof of the recurrence of the ERRW in all
dimensions at strong reinforcement.
In a paper in progress, \cite{SabTar} make a link between the
VRJP and accurate pointwise large deviation for reversible Markov jump
processes: it appears that the limit measure of the VRJP \cite{SabTar11} is closely related to the first order of pointwise large
deviations which are derived in \cite{SabTar} for continuous time Markov
processes using the present local limit theorem and Remark~\ref{main-unif}.

\noindent \textbf{Notations.} Any vector $v\equiv(v_k)\in\C^N$ is considered as a row-vector and $v^{\top}$ is the corresponding column-vector. The vector with all components equal to 1 is denoted by $\mathbf{1}$. The Euclidean scalar product  and its associated norm  on $\C^N$ is denoted by $\langle\cdot,\cdot\rangle$ and $\| \cdot\|$ respectively. 
The set of $N\times N$-matrices with complex entries is denoted by $\cM_N(\C)$. 
We use the following norm $\|\cdot\|_0$ on $\cM_N(\C)$:  
$$\forall A\equiv (A_{k,\ell})\in\cM_N(\C),\quad \|A\|_0:= \max\big\{|A_{k,\ell}|: (k,\ell)\in\{1,\ldots,N\}^2 \big\}.$$ 
For any bounded positive measure $\nu$ on $\R^d$, we define its Fourier transform as: 
$$\forall \zeta\in\R^d,\quad \widehat\nu(\zeta) := \int_{\R^d}e^{i \langle \zeta , y \rangle}\, d\nu(y).$$
Let  $\cA\equiv (\cA_{k,\ell})$ be a $N\times N$-matrix with entries in the set of bounded positive measures on $\R^d$. We set 
\begin{equation}\label{def-Four-mat}
\forall B\in B(\R^d), \quad \cA(1_B) := \big(\cA_{k,\ell}(1_B)\big), \qquad
\forall \zeta\in\R^d, \quad \widehat \cA(\zeta) := (\widehat\cA_{k,\ell}(\zeta)). 
\end{equation}

\section{The LLT for the density process} \label{sec-ass-stat}
Let $\{(X_t,Z_t)\}_{t\in \T}$ be an MAP with state space $\X \times \R^d$, where  $\X:=\{1,\ldots,N\}$ and the driving Markov process $\{X_t\}_{t\in\T}$ has transition semi-group  $\{P_t\}_{t\in\T}$. We refer to \cite[Chap. XI]{Asm03} for the basic material on such MAPs. The conditional probability to $\{X_0=k\}$ and its associated expectation are denoted by $\P_k$ and $\E_k$ respectively. Note that if $T :\R^d\r\R^m$ is a linear transformation, then $\{X_t,T(Z_t)\}_{t\in\T}$ is still a MAP on $\X\times R^m$ (see Lem. C.1 in [HL]). We suppose that  $\{X_t\}_{t\in\T}$ has a unique invariant probability measure $\pi$. Set $m:= \E_{\pi}[Z_1]=\sum_k \pi(k) \E_{k}[Z_1]\in\R^d$. Consider the centered MAP $\{(X_t,Y_t)\}_{t\in \T}$ where $Y_t := Z_t - t\, m$. The two next assumptions are involved in both CLT and LLT below. \\
\vspace*{-4mm}

\noindent \IP ~: 
{\it The stochastic $N\times N$-matrix $P:=P_1$ is irreducible and aperiodic.}

\noindent $\M{\alpha}$ : 
{\it 
The family of r.v.~$\{Y_v\}_{v\in(0,1]\cap\T}$ satisfies the  uniform moment condition of order $\alpha$:} 
\begin{equation} \label{moment-alpha} 
M_{\alpha}:= \max_{k\in\X} \sup_{v\in(0,1]\cap\T} \E_k\big[\|Y_v\|^{\alpha}\big] < \infty.  
\end{equation}
The next theorem provides a CLT for $t^{-1/2}Y_t$, proved when $d:=1$ in \cite{KeiWis64} for $\T=\N$ and in \cite{FukHit67} for $\T=[0,\infty)$; see  \cite{FerHerLed12} for $\rho$-mixing driving Markov processes. 
\begin{theo} \label{lem-clt}
Under Assumptions~\emph{\IP}~and $\M{2}$, $\{t^{-1/2}Y_t\}_{t\in \T}$ converges in distribution to a $d$-dimensional Gaussian law $\cN(0,\Sigma)$ when $t\r+\infty$. 
\end{theo} 
\begin{rem} When $\T=\N$, the moment condition reads as: $\E_\pi[\|Y_1\|^2] < \infty$. When $\T:=[0,\infty)$, since $\X$ is a finite set, we know that Condition~$\M{2}$ on the family $\{Y_v\}_{v\in (0,1] \cap \T}$ is still equivalent to $\E_\pi[\|Y_1\|^2] < \infty$  \emph{\cite{FukHit67}}. Moreover, note that \emph{\IP} is satisfied  in the continuous time case provided that the generator of $\{X_t\}_{t\ge 0}$ is irreducible. 
\end{rem}

Now let us specify the notations and assumptions involved in our LLT. 
For any $t\in\T$ and $(k,\ell)\in\X^2$, we define the  bounded positive measure $\cY_{k,\ell,t}$ on $\R^d$ as follows (see (\ref{Def_Add})): 
\begin{equation} \label{loi-Yt-kl}
\forall B\in B(\R^d),\quad \cY_{k,\ell,t}(1_B) := \P_k\big\{X_t=\ell, Y_t\in B\big\} = Q_t(k,0;\{\ell\}\times B).  
\end{equation}
Let $\ell_d$ denote the Lebesgue measure on $\R^d$. From the Lebesgue decomposition of $\cY_{k,\ell,t}$ w.r.t.~$\ell_d$, there are two bounded positive measures $\cG_{k,\ell,t}$ and $\mu_{k,\ell,t}$ on $\R^d$ such that
\begin{equation} \label{Leb-kl}
\forall B\in B(\R^d), \quad  \cY_{k,\ell,t}(1_B):= \cG_{k,\ell,t}(1_B) + \mu_{k,\ell,t}(1_B) 
\ \text{ where }    \cG_{k,\ell,t}(1_B)=\int_B g_{k,\ell,t}(y)\, dy   
\end{equation}
for some measurable function $g_{k,\ell,t} : \R^d\r[0,+\infty)$, and such that $\mu_{k,\ell,t}$ and $\ell_d$ are mutually singular. The  measure $\cG_{k,\ell,t}$ is called the absolutely continuous (a.c.) part of  $\cY_{k,\ell,t}$ with associated density $g_{k,\ell,t}$. 
For any $t\in\T$, we introduce the following $N\times N$-matrices with entries in the set of bounded positive measures on $\R^d$ 
\begin{subequations}
\begin{equation} \label{def-meas-mat} 
 \cY_{t} := (\cY_{k,\ell,t})_{(k,\ell)\in\X^2},\quad  \cG_t:= (\cG_{k,\ell,t})_{(k,\ell)\in\X^2}, \quad  \cM_t := (\mu_{k,\ell,t})_{(k,\ell)\in\X^2}, 
\end{equation}
and for every $y\in\R^d$, we define the following real $N\times N$-matrix: 
\begin{equation} \label{def-Gt-density}
G_t(y) := \big(g_{k,\ell,t}(y)\big)_{(k,\ell)\in\X^2}.
\end{equation}
Then the component-wise equalities (\ref{Leb-kl}) read as follows in a matrix form: for any $t\in\T$  
\begin{equation} \label{dec-leb-mat}
\forall B\in B(\R^d), \quad \cY_{t}(1_B) = \cG_{t}(1_B) + \cM_{t}(1_B) =\int_B G_t(y) dy + \cM_t(1_B). 
\end{equation}
\end{subequations}
The assumptions on the a.c.~part $\cG_t$ and the singular part $\cM_t$ of $\cY_t$ are the following ones.
\begin{ass}: \label{AS2}
There exist $c>0$ and $\rho\in(0,1)$ such that   
\begin{equation} \label{masse-id}
\forall t >0,\quad  \| \cM_{t}(1_{\R^d}) \|_{0}  \leq c\rho^t 
\end{equation}
and there exists $t_0>0$ such that $\rho^{t_0}\max(2,cN)\leq 1/4$ and 
\begin{equation} \label{fourier-id}
\Gamma_{t_0}(\zeta) :=  \sup_{w\in[t_0,2t_0)} \| \widehat G_{w}(\zeta)\|_0 \longrightarrow 0\quad \text{when}\ \|\zeta\|\r+\infty.
\end{equation}
\end{ass}
\vspace*{-4mm}
 \begin{ass}: \label{AS0} For any $t>0$, there exists an open convex subset $\cD_t$ of $\R^d$ such that $G_t$ vanishes on $\R^d\setminus\overline{\cD}_t$, where $\overline{\cD}_t$ denotes the closure of $\cD_t$. Moreover $G_t$ is continuous on $\overline{\cD}_t$ and differentiable on $\cD_t$, with in addition 
\begin{subequations}
\begin{gather*} 
\sup_{t>0} \sup_{y\in \overline{{\cal D}}_t} \| G_t(y)\|_0 <\infty  \qquad \sup_{y\in \partial {\cal D}_t}  \|G_t(y)\|_0 = O\big(t^{-(d+1)/2}\big)  \quad \text{where}\ \ \partial {\cal D}_t := \overline{{\cal D}}_t\setminus \cD_t  \\ 
 j=1,\ldots,d:\quad \sup_{t>0}\sup_{y\in {\cal D}_t} \big\| \frac{\partial G_t}{\partial y_j}(y)\big\|_0 <\infty.  
\end{gather*}
\end{subequations}
\end{ass}

The next theorem is the main result of the paper. 
\vspace*{-2mm}
\begin{theo} \label{main-id}
Let $\{(X_t,Y_t)\}_{t\in \T}$ be a centered MAP 
satisfying Assumptions~\emph{\IP}, $\M{3}$, \emph{\textbf{(AC\ref{AS2})-(AC\ref{AS0})}}. Moreover assume that the matrix $\Sigma$ associated with the CLT in Theorem~\ref{lem-clt} is invertible. Then, 
for every $k\in\X$, the density $f_{k,t}(\cdot)$ of the a.c.~part   of the probability distribution of $t^{-1/2}Y_t$ under $\P_k$ satisfies the following property: 
\begin{equation} \label{result_final}
\sup_{y \in \R^d} \big|f_{k,t}(y) - \eta_{_\Sigma}(y)\big| \leq O\big( t^{-1/2}\big) + O\big(\sup_
{y\notin {\cal D}_t}\eta_{_\Sigma}(t^{-1/2}y) \big)
\end{equation}
where $\eta_{_\Sigma}(\cdot)$ denotes the density of $\cN(0,\Sigma)$.  
\end{theo}
\begin{rem}[Non-lattice condition] \label{rque-NL}
To derive a LLT, it  is standard to assume some non-lattice condition which writes for the MAP $\{(X_t,Y_t)\}_{t\in \T}$ as (e.g~see \emph{\cite{FerHerLed12}}): 

\medskip

\noindent \emph{\NL} : {\it  there is no triplet $(a,F,\theta)$  with $a\in\R^d$, $F$ a closed subgroup of $\R^d$,  $F\neq \R^d$ and $\beta\, :\, \X\r\R^d$ such that: 
$\forall k\in\X,\quad Y_1 + \beta(X_1) - \beta(k)\ \in\ a + F\ \ \P_{k}\text{-a.s.}$.}

\medskip 

\noindent Actually Condition \emph{\NL} holds 
under \emph{\IP} and  \emph{\textbf{(AC\ref{AS2})}} (see the proof of Lemma~\ref{lem-non-ari}). Moreover, when $\{Y_t\}_{t\geq0}$ is an additive functional of $\{X_t\}_{t\geq0}$, the  matrix $\Sigma$ in the CLT for $\{t^{-1/2}Y_t\}_{t>0}$ is  invertible under \emph{\NL}. Indeed, if not, there is $\zeta_0\in\R^{N}$ such that $\langle \zeta_0,\Sigma\zeta_0\rangle =0$. Thus,  $\{t^{-1/2}\langle \zeta_0,Y_t\rangle\}_{t>0}$ converges in distribution to $0$ (which is a degenerate Gaussian random variable).  
It follows from \emph{\cite[Th.~2.2]{HitShi70}} that, for every $k\in\X$, $\langle \zeta_0,Y_1\rangle = b(X_1) - b(k)\ \P_k$-a.s. for some $b : \X\r\R$. But this is impossible under~\emph{\NL}.  
\end{rem}
\begin{rem}[A uniform LLT with respect to transition matrix $P$] \label{main-unif}
Let $\cP$ denote the set of irreducible and aperiodic stochastic $N\times N$-matrices, equipped with the topology associated with (for instance) the distance $d(P,P') := \|P-P'\|_{0}\, $ ($P,P'\in\cP$). 
Assume that there is some compact subset $\cP_0$ of $\cP$ such that $P:=P_1 \in \cP_0$, where $\{P_t\}_{t\in\T}$ is the transition semi-group of $\{X_t\}_{t\in\T}$. 
To keep in mind the dependence on $P$, the positive measures $\cY_{k,\ell,t}$ in (\ref{loi-Yt-kl}) and the matrix $\Sigma$ in Theorem~\ref{lem-clt} are denoted by $\cY_{k,\ell,1}^{P}$ and $\Sigma^P$ respectively. Finally let $(\mathfrak{M},d_{TV})$ be the space of bounded positive measures on  $\R^d$ equipped with the total variation distance $d_{TV}$ (i.e.~$\forall (\mu,\mu')\in \mathfrak{M}^2,\ d_{TV}(\mu,\mu') := \sup\{|\mu(f)-\mu'(f)|,\, |f|\leq1\}$. Let us assume that the centered MAP $\{(X_t,Y_t)\}_{t\in \T}$ satisfies the following conditions:
\begin{assu}: \label{U1}
For any $(k,\ell)\in \X^2$, the map $P\mapsto \cY_{k,\ell,1}^{P}$ is continuous from $(\cP_0,d)$ into $(\mathfrak{M},d_{TV})$. 
\end{assu}
\begin{assu}: \label{U4} 
There exist positive constants $\alpha$ and $\beta$ such that 
\begin{equation} \label{alpha-beta} 
\forall P\in\cP_0,\ \forall \zeta\in\R^d,\quad \alpha\, \|\zeta\|^2 \leq \langle \zeta,\Sigma^P \zeta \rangle \leq \beta\, \|\zeta\|^2. 
\end{equation} 
\end{assu}
\begin{assu}: \label{U3} 
The conditions $\M{3}$, \emph{\textbf{(AC\ref{AS2})}} and \emph{\textbf{(AC\ref{AS0})}} hold uniformly in $P\in \cP_0$.
\end{assu}
Then, under Assumptions~\emph{\textbf{(U\ref{U1})-(U\ref{U3})}}, for every $k\in \X$, the density $f_{k,t}^{P}(\cdot)$ of the a.c.~part of the probability distribution of $t^{-1/2}Y_t$ under $\P_k$ satisfies the following 
property when $t\r+\infty$: 
$$\sup_{P\in{\cal P}_0} \sup_{y\in\R^d}\big|f_{k,t}^{P}(y) - \eta_{_{\Sigma^P}}(y)\big| = O(t^{-1/2}) + O\big(\sup_{P\in{\cal P}_0}\sup_
{y\notin {\cal D}_t}\eta_{_{\Sigma^P}}(t^{-1/2}y) \big).$$
This result follows from a suitable adaptation of the proof of Theorem~\emph{\ref{main-id}} (see Remark~\emph{\ref{rque-unif}}).
\end{rem}
\section{Application to the local times of a jump process} \label{appli-local-times}

Let $\{X_t\}_{t\ge 0}$ be a Markov jump process  with finite state space $\X:=\{1,\ldots,N\}$ and generator $G$. Its transition semi-group is given by: $\forall t\ge 0,\ P_t := e^{tG}$. The local time $L_t(i)$ associated with state $i\in\X$, or the sojourn time in state $i$ on the interval $[0,t]$,  is defined by 
	\[ \forall t\ge 0, \quad L_t(i):= \int_0^t 1_{\{X_s=i\}} \, ds.
\]
It is well known that $L_t(i)$ is an additive functional of $\{X_t\}_{t\ge 0}$ and that $\{(X_t,L_t(i))\}_{t\ge 0}$ is an MAP. Consider the MAP $\{(X_t,L_t)\}_{t\ge 0}$ where $L_t$ is the random vector of the local times 
$L_t:= (L_t(1),\ldots,L_t(N))$. 
 Note that, for all $t>0$, we have $\langle L_t, \mathbf{1}\rangle =t$, that is $L_t$ is $\cS_t-$valued where 
$$\cS_t := \big\{ y\in [0,+\infty)^N: \langle y, \mathbf{1}\rangle =t\big\}.$$ 
Assume that  $G$ is irreducible. Then Condition \IP~holds true and $\{X_t\}_{t\ge 0}$ has a unique invariant probability measure $\pi$. Set $m= (m_1,\ldots,m_N):=\E_{\pi}[L_1]$, and define the $\cS_{t}^{(0)}-$valued centered r.v. $$Y_t = L_t - tm$$ where $\cS_{t}^{(0)}:=T_{-tm}(\cS_t)$ with the translation $T_{-tm}$ by vector $-tm$ in $\R^N$. Note that  $\cS_{t}^{(0)}$ is a subset of the hyperplane $H$ of $\R^N$ defined by  
\begin{equation} \label{Def_H_Local}
H:=\big\{y\in\R^N: \ \langle y, \mathbf{1}\rangle = 0\big\}.
\end{equation}
We denote by $\Lambda$ the bijective map from $H$ into $\R^{N-1}$ defined by $\Lambda(y):=(y_1,\ldots,y_{N-1})$ for $y:=(y_1,\ldots,y_N)\in H$. Recall that $\ell_{N-1}$ denotes the Lebesgue measure on $\R^{N-1}$. Let $\nu$ be the measure defined on $(H,B(H))$ as the image measure of $\ell_{N-1}$ under $\Lambda^{-1}$, where $B(H)$ stands for the Borelian $\sigma$-algebra on $H$. Finally, under the probability measure $\P_k$, $f_{k,t}$ is  
the density  of the a.c.~part of the probability distribution of $t^{-1/2}Y_t = t^{-1/2}(L_t-tm)$ with respect to the measure $\nu$ on $H$. 
\begin{pro} \label{prop-local-times}
If $G$ and the sub-generators $G_{i^ci^c}:=(G(k,\ell))_{k,\ell\in\{i\}^c}, \,i=1,\ldots,N$ are irreducible, then there exists a definite-positive $(N-1)\times(N-1)$-matrix $\Sigma$ such that:
$$\forall k\in \X, \qquad\sup_{h \in H} \big|f_{k,t}(h) - (2\pi)^{-(N-1)/2}(\det \Sigma)^{-1/2} e^{-\frac{1}{2}\langle \Lambda h^{\top},\Sigma \Lambda h^{\top} \rangle}\big| = O\big(t^{-1/2}\big).$$
\end{pro}
\begin{proof}{}
Let $Y'_t$ be the $(N-1)$-dimensional random vector  
$\Lambda(Y_t) = (Y_t(1),\ldots,Y_t(N-1))$. 
Then $\{(X_t,Y'_t)\}_{t\ge 0}$ is a $\X\times\overline{\cD}_t-$valued MAP where $\cD_t$ is the following open convex of $\R^{N-1}$ 
\begin{equation} \label{Def_Dt}
	\cD_t := \big\{ y'\in \R^{N-1} : j=1,\ldots, N-1, \ y'_j \in (-m_j t, (1-m_j)t), \ \langle y',\mathbf{1}\rangle < m_N t \big\}.
\end{equation} 
Using the assumptions on $G$ and results of \cite{Ser00}, the MAP $\{(X_t,Y'_t)\}_{t\ge 0}$ satisfies the conditions $\M{3}$ and \textbf{(AC1)-(AC2)} (see details in Section~5 of [HL]). The matrix $\Sigma$ in the CLT for $\{t^{-1/2}Y'_t\}_{t>0}$ is invertible from Remark~\ref{rque-NL} and  $O(\sup_{y\notin {\cal D}_t}\eta_{_\Sigma}(t^{-1/2}y)) = O(t^{-1/2})$ from (\ref{Def_Dt}). Thus Theorem~\ref{main-id} gives that 
$$\forall k\in\X,\quad \sup_{y' \in \R^{N-1}} \big|f'_{k,t}(y') - (2\pi)^{-(N-1)/2}(\det \Sigma)^{-1/2} e^{-\frac{1}{2}\langle {y'}^{\top},\Sigma {y'}^{\top} \rangle}\big| = O\big(t^{-1/2}\big),$$
where $f'_{k,t}$ is the density of the a.c.~part of the probability distribution of $t^{-1/2}Y'_t$. Finally, using the bijection $\Lambda$, it is easily seen that the density $f_{k,t}$ is given by: $\forall h\in H,\ f_{k,t}(h) := f'_{k,t}(\Lambda h)$. This gives the claimed result. 
\end{proof}
\begin{rem} First, the assumption on the sub-generators of $G$ in Proposition~\ref{prop-local-times} is used to obtain the geometric convergence to 0 of the mass of the singular part of the probability distribution of $Y'_t$ required in \emph{\textbf{(AC\ref{AS2})}} (see Section~5 of [HL] for details). Second, it can be seen from \emph{\cite{Pin91}} that $\Sigma$ is the  definite-positive matrix  
	\[ \forall (i,j)\in\{1,\ldots,N-1\}^2, \quad \Sigma_{i,j} = \pi_i C_{i,j} + \pi_j C_{j,i} \quad \text{ where } \quad C:=-\int_0^{\infty} (e^{tG} - \mathbf{1}^{\top} \pi)\, dt. 
\]
\end{rem}

\section{Fourier analysis and proof of Theorem~\ref{main-id}}
Assume that the conditions of Theorem~\ref{main-id} hold. Note that $\{(X_t,\Sigma^{-1/2} Y_t)\}_{t\in\T}$ is an MAP with the identity matrix $I$ as asymptotic covariance matrix. It still satisfies the assumptions of Theorem~\ref{main-id} since the associated bounded positive measures of (\ref{Leb-kl}) are the image measure of $\cG_{k,\ell,t}$ and $\cM_{k,\ell,t}$ under the bijective linear map $\Sigma^{-1/2}$. 
Thus, we only have to prove  Theorem~\ref{main-id} when $\Sigma=I$ . This proof 
 involves Fourier analysis as in the i.i.d.~case. In our Markov context, this study is based on the semi-group property of the matrix family $\{\widehat\cY_t(\zeta)\}_{t\in\T}$ for every $\zeta\in\R^d$ which allows us to analyze the characteristic function of $Y_t$. In the next subsection,  we provide a collection of lemmas which highlights the connections between the assumptions of Section~\ref{sec-ass-stat} and the behavior of this semi-group.
This is the basic material for the derivation of Theorem~\ref{main-id} in Subsection~\ref{Main}.

Let us denote the integer part of any $t\in\T$ by $\lfloor t \rfloor$. 
Using the notation of (\ref{Leb-kl}), the bounded positive measure $\sum_{\ell=1}^N \cG_{k,\ell,t}$ with density 
$g_{k,t}:=\sum_{\ell=1}^N g_{k,\ell,t}$ 
is the a.c.~part of the probability distribution of $Y_t$ under $\P_k$, while $\mu_{k,t}:=\sum_{\ell=1}^N\mu_{k,\ell,t}$ is its singular part. 
That is, we have for any $k\in\X$ and $t>0$: 
\begin{equation} \label{loi-Yt}
\forall B\in B(\R^d),\quad \P_k\{Y_t\in B\} = \int_B g_{k,t}(y)\, dy + \mu_{k,t}(1_B).
\vspace*{-3mm}
\end{equation}
\subsection{Semigroup of Fourier matrices and basic lemmas}
 The bounded positive measure  $\cY_t$ is defined in (\ref{def-meas-mat}) and its Fourier transform $\widehat\cY_t$ is (see (\ref{def-Four-mat})): 
\begin{equation} \label{Fourier}
\forall (t,\zeta)\in\T\times\R^d,\ \forall (k,\ell)\in \X^2,\quad \big(\widehat\cY_t(\zeta)\big)_{k,\ell} = \E_{k}\big[1_{\{X_t=\ell\}}\, e^{i \langle \zeta , Y_t \rangle}\big]. 
\end{equation}
Note that $\widehat\cY_t(0) = P_t$. From the additivity of the second component $Y_t$, we know that  $\{\widehat\cY_t(\zeta)\}_{t\in\T,\zeta\in\R^d}$ is a semi-group of matrices (e.g. see \cite{FerHerLed12} for details), that is 
\begin{equation} \label{semi-group}
\forall \zeta\in\R^d,\ \forall(s,t)\in\T^2,\quad \widehat\cY_{t+s}(\zeta) = \widehat\cY_t(\zeta)\, \widehat\cY_s(\zeta). \tag{SG}
\end{equation}
In particular the following property holds true
\begin{equation} \label{semi-gpe-discret}
\forall \zeta\in\R^d,\ \forall n\in\N,\qquad \widehat\cY_n(\zeta) := \big(\E_{k}[1_{\{X_n=\ell\}}\, e^{i \langle \zeta , Y_n \rangle}]\big)_{(k,\ell)\in\X^2} = \widehat\cY_1(\zeta)^n. 
\end{equation}
For every $k\in\X$ and $t\in\T$, we denote by $\phi_{k,t}$ the characteristic function of $Y_t$ under $\P_k$: 
\begin{equation} \label{phi-Y}
\forall \zeta\in\R^d,\quad \phi_{k,t}(\zeta) := \E_{k}\big[e^{i \langle \zeta , Y_t \rangle}\, \big] = e_k \widehat\cY_t(\zeta) \mathbf{1}^{\top},
\end{equation}
where $e_k$ is the $k$-th vector of the canonical basis of $\R^d$. 

Under Conditions \IP and $\M{2}$, the spectral method provides the CLT for $\{(t^{-1/2}Y_t)\}_{t\in\T}$. To extend this CLT to a local limit theorem, a precise control of the characteristic function $\phi_{k,t}$ is  needed.  This is the purpose of the next three lemmas, in which the semi-group property (\ref{semi-group}) plays an important role. The first lemma, which is the central part in the proof of Theorem~\ref{main-id},  provides the control of $\phi_{k,t}$ on a ball $B(0,\delta):= \{\zeta\in\R^d : \|\zeta\| < \delta\}$ for some $\delta>0$. The second one is on the control of $\phi_{k,t}$ on the annulus $\{\zeta\in\R^d : \delta\leq\|\zeta\|\leq A]$ for any $A>\delta$. Finally the third lemma focuses on the Fourier transform of the density $g_{k,\ell,t}$ of the a.c. part $\cG_{k,\ell,t}$ of $\cY_{k,\ell,t}$ in (\ref{Leb-kl}) on the domain $\{\zeta\in\R^d :\|\zeta\|\geq A\}$ for some $A>0$.  
\begin{lem} \label{lem-dec-vp}
Under Assumptions~\emph{\IP}~and $\M{3}$, there exists a real number $\delta>0$ 
such that, for all $\zeta\in B(0,\delta)$, the characteristic function of $Y_t$ satisfies 
\begin{equation} \label{exp-val-pert}
\forall k\in\X,\ \forall t\in\T,\quad \phi_{k,t}(\zeta)  = \lambda(\zeta)^{\lfloor t \rfloor}\, L_{k,t}(\zeta)  + R_{k,t}(\zeta),
\end{equation}
with $\C-$valued functions $\lambda(\cdot)$, $L_{k,t}(\cdot)$ and $R_{k,t}(\cdot)$ on $B(0,\delta)$ satisfying the next properties for $t\in\T$, $\zeta\in\R^d$ and $k\in\X$: 
\begin{subequations}
\begin{eqnarray}
 \|\zeta\| < \delta\ &\Rightarrow & \lambda(\zeta) = 1 - \|\zeta\|^2/2 + O(\|\zeta\|^3) \label{exp-val-pert0} \\
 t\geq 2,\ \|t^{-1/2}\zeta\| < \delta & \Rightarrow &  \big|\lambda(t^{-1/2}\zeta)^{\lfloor t \rfloor} - e^{-\|\zeta\|^2/2}\big| \leq t^{-1/2} C\,(1+\|\zeta\|^3)e^{-\|\zeta\|^2/8} \label{exp-val-pert0-bis} \\
 \|\zeta\| < \delta & \Rightarrow &\big|L_{k,t}(\zeta)-1\big| \leq C\, \|\zeta\|  \label{exp-val-pert1} \\
 \exists\, r\in(0,1), & &  \sup\{|R_{k,t}(\zeta)|: \|\zeta\|\leq \delta \} \leq C\, r^{\lfloor t \rfloor}  \label{exp-val-pert2}
\end{eqnarray}
\end{subequations}
where the constant $C>0$ in \emph{(\ref{exp-val-pert0-bis})}-\emph{(\ref{exp-val-pert2})} only depends on $\delta$ and on $M_3$ in $\M{3}$ (see \emph{(\ref{moment-alpha})}). 
\end{lem}
\begin{proof}{}
Assumption~\IP~ensures that $\widehat\cY_1(0) = P_1 = \Pi + N$, where $\Pi:=(\Pi_{i,j})_{(i,j)\in \X}$ is the rank-one matrix defined by 
$\Pi_{i,j}:=\pi(i)$ and $N\in\cM_N(\C)$ is such that $\|N^n\|_0 = O(\kappa^n)$ for some $\kappa\in(0,1)$. Moreover note that the map $\zeta\mapsto \widehat\cY_1(\zeta)$ is thrice-differentiable from $\R^d$ into $\cM_N(\C)$ due to (\ref{Fourier}) and $\M{3}$. Then the standard  perturbation theory shows that, for any $r\in(\kappa,1)$, there exists $\delta\equiv\delta(r)>0$ such that, for all $\zeta\in B(0,\delta)$, $\widehat\cY_1(\zeta)^n = \lambda(\zeta)^n\Pi(\zeta) + N(\zeta)^n$ where $\lambda(\zeta)$ is the dominating eigenvalue of $\widehat\cY_1(\zeta)$, $\Pi(\zeta)$ is the associated rank-one eigen-projection, and $N(\zeta)\in\cM_N(\C)$ is such that $\|N(\zeta)^n\|_0 = O(r^n)$. Moreover the maps $\lambda(\cdot)$, $\Pi(\cdot)$, and $N(\cdot)$ are thrice-differentiable on $B(0,\delta)$. These properties and (\ref{phi-Y}) provide all the conclusions of Lemma~\ref{lem-dec-vp} for $t\in\N$ (see \cite[Prop.~VI.2]{HenHer01} for details). For Properties (\ref{exp-val-pert0}), (\ref{exp-val-pert1}) and (\ref{exp-val-pert2}), the passage to the continuous-time case $\T:=[0,+\infty)$ can be easily derived from 
$\phi_{k,t}(\zeta) = 
e_k \widehat\cY_1(\zeta)^{\lfloor t \rfloor} \widehat\cY_v(\zeta) \mathbf{1}^{\top}$ 
with $v:=t-\lfloor t \rfloor\in[0,1)$ 
and the fact that $\zeta\mapsto \widehat\cY_v(\zeta)$ is thrice-differentiable from $\R^d$ into $\cM_N(\C)$ with  partial derivatives uniformly bounded in $v\in(0,1]$. 
See \cite[Prop.~4.4]{FerHerLed12} for details. For (\ref{exp-val-pert0-bis}), the passage to $\T:=[0,+\infty)$ is an easy extension of Inequality (v) of \cite[Prop.~VI.2]{HenHer01}) (see details in [HL]). 
\end{proof}
\begin{lem} \label{lem-non-ari} 
Assume that Conditions~\emph{\IP}, \emph{\textbf{(AC1)}} and $\M{\alpha}$ for some $\alpha>0$ hold. Let $\delta,A$ be any real numbers such that $0<\delta<A$. There exist constants $D\equiv D(\delta,A)>0$ and $\tau\equiv \tau(\delta,A)\in(0,1)$ such that 
\begin{equation} \label{ineg-non-ari}
\forall k\in\X,\ \forall t\in\T,\quad \sup\big\{|\phi_{k,t}(\zeta)|: \delta\leq\|\zeta\|\leq A\big\} \leq D\, \tau^{\lfloor t \rfloor}.
\end{equation}
\end{lem}
\begin{proof}{} Lemma~\ref{lem-non-ari} can be classically derived from the spectral method under the non-lattice condition \NL~of Remark~\ref{rque-NL} and moment condition  $\M{\alpha}$ for some $\alpha>0$  \cite[p.~412]{FerHerLed12} (see also the proof of Lem. 6.2 in [HL] for details). Therefore, it is enough to prove that $\{Y_t\}_{t\in\T}$ satisfies \NL~under Assumptions~{\IP}~and \textbf{(AC1)}. 

For $t$ large enough, the function $g_{k,t}$ is nonzero in the Lebesgue space $\L^1(\R^d)$ since the mass $\mu_{k,t}(1_{\R^d})$ goes to $0$ when $t\r+\infty$ from Assumption~\textbf{(AC1)}. Let us fix $k\in\X$ and some integer $q\in\N^*$ such that $g_{k,q}\neq 0$ in $\L^1(\R^d)$.  
Assume that $\{Y_t\}_{t\in\T}$ is lattice. Then there exist $a\in\R^d$, $\theta : \X\r\R^d$ and a closed subgroup $F$ in $\R^d$, $F\neq \R^d$, such that 
$$Y_{q} + \theta(X_{q}) - \theta(k) \ \in \   a+F \ \ \ \P_k-\text{a.s.}\footnote{The lattice condition is classically expressed under $\P_\pi$ with $q=1$ by using the spectral method. However this condition can be similarly expressed with $Y_q$ for any $q\in\N^*$. Moreover it can be stated under $\P_k$ for every $k\in\X$ since $\X$ is finite and $P$ is irreducible and aperiodic from \IP.} .$$
For any $\ell\in\X$, define the following subset of $\R^d$: $F_{k,\ell} := \theta(k) - \theta(\ell) + a+F$. Note that $\ell_d(F_{k,\ell})=0$ since $F$ is a proper closed subgroup of $\R^d$. Since $\mu_{k,\ell,q}$ and $\ell_d$ are singular, there is $E_{k,\ell}\in B(\R^d)$ such that $\ell_d(E_{k,\ell})=0$ and $\mu_{k,\ell,q}(B)=\mu_{k,\ell,q}(B\cap E_{k,\ell})$ for any $B\in B(\R^d)$. Set $E_k:=\cup_{j=1}^N E_{k,j}$ and ${E_k}^c:=\R^d\setminus E_k$. Observe that $\mu_{k,\ell,q}({E_k}^c)=0$. We obtain  
\begin{eqnarray*}
\forall B\in B(\R^d), \ \P_k\big\{Y_q\in B\cap E_k^c\big\} &=& \sum_{\ell=1}^N \P_k\big\{X_q=\ell,Y_q\in F_{k,\ell}\cap B\cap E_k^c\big\} \\
&=& \sum_{\ell=1}^N\int_{F_{k,\ell}\cap B\cap E_k^c}g_{k,\ell,q}(y)\, dy + \sum_{\ell=1}^N\mu_{k,\ell,q}(1_{F_{k,\ell}\cap B\cap E_k^c})  =0.
\end{eqnarray*}
On the other hand, we have using (\ref{loi-Yt}) 
$$\forall B\in B(\R^d), \quad \P_k\big\{Y_q\in B\cap E_k^c\big\} = \int_{B\cap E_k^c}g_{k,q}(y)\, dy + \mu_{k,q}(1_{B\cap E_k^c}) = \int_{B\cap E_k^c}g_{k,q}(y)\, dy.$$
From these equalities and $\ell_d(E_k)=0$, it follows that  
$\int_{B}g_{k,q}(y)\, dy = \int_{B\cap E_k^c}g_{k,q}(y)\, dy = 0$. 
But this is impossible since $ g_{k,q}\neq 0$ in $\L^1(\R^d)$.  
\end{proof}
\begin{lem} \label{lem-hat-g} 
Under Condition~\emph{\textbf{(AC\ref{AS2})}}, there exist positive constants $A$ and $C$ such that  
$$|\zeta|\geq A\ \Longrightarrow\ \forall (k,\ell)\in\X^2,\ \forall t\in [t_0,+\infty), 
\quad \big|\widehat g_{k,\ell,t}(\zeta)\big| \leq C\, \frac{t}{2^{t/t_0}}.$$
\end{lem}
\begin{proof}{ of Lemma~\emph{\ref{lem-hat-g}}}  Let us introduce the following matrix norm: $\forall A\in\cM_N(\C),\quad \|A\|_\infty := \sup_{\|v\|_0=1}\|A v^{\top}\|_0$. 
It is easily seen that 
\begin{equation} \label{equiv-norm}
\forall A\in\cM_N(\C),\quad \|A\|_0\leq \|A\|_\infty \leq N \|A\|_0.
\end{equation}
Let $(t,\zeta)\in\T\times\R^d$. From (\ref{dec-leb-mat}) we have $\widehat\cY_t(\zeta) = \widehat \cG_t(\zeta) + \widehat \cM_t(\zeta)$ and the semi-group property (\ref{semi-group}) 
gives for any $(s,t)\in\T^2$ and $\zeta\in\R^d$:  
\begin{eqnarray*}
\widehat\cY_{t+s}(\zeta) := \widehat \cG_{t+s}(\zeta) +  \widehat \cM_{t+s}(\zeta)  
&=& \widehat\cY_t(\zeta)\, \widehat\cY_s(\zeta) \label{car-int}  \\
&=& \big(\widehat \cG_t(\zeta) + \widehat \cM_t(\zeta)\big)\big(\widehat \cG_s(\zeta) +  \widehat \cM_s(\zeta)\big)  \\[0.12cm]
&=& \widehat \cG_t(\zeta)\widehat \cG_s(\zeta) + \widehat \cG_t(\zeta) \widehat \cM_s(\zeta) +  \widehat \cM_t(\zeta) \widehat  \cG_s(\zeta) + \widehat \cM_t(\zeta) \widehat \cM_s(\zeta).
\end{eqnarray*}
Thus
$$\widehat \cG_{t+s}(\zeta) = \widehat \cG_t(\zeta)\widehat \cG_s(\zeta) + \widehat \cG_t(\zeta) \widehat \cM_s(\zeta) +  \widehat \cM_t(\zeta) \widehat  \cG_s(\zeta) + \widehat \cM_t(\zeta) \widehat \cM_s(\zeta) - \widehat \cM_{t+s}(\zeta).$$
Then, we obtain using the matrix norm $\|\cdot\|_\infty$:  
\begin{eqnarray}
\|\widehat \cG_{t+s}(\zeta)\|_\infty &\leq& \|\widehat \cG_t(\zeta)\|_\infty \|\widehat \cG_s(\zeta)\|_\infty + \|\widehat \cG_t(\zeta)\|_\infty \|\widehat \cM_s(\zeta)\|_\infty + \|\widehat \cM_t(\zeta)\|_\infty \|\widehat \cG_s(\zeta)\|_\infty \nonumber \\
&\ & \qquad  \qquad  \qquad  \qquad + \qquad \|\widehat \cM_t(\zeta)\|_\infty \| \widehat \cM_s(\zeta)\|_\infty +  \|\widehat \cM_{t+s}(\zeta)\|_\infty. \label{hat-G-presque-semi}
\end{eqnarray}
Moreover we deduce from (\ref{masse-id}) and (\ref{equiv-norm}) with $\rho$ and $c$ given in Condition~\textbf{(AC\ref{AS2})}, that: 
\begin{equation} \label{masse-id-bis}
\forall u>0,\ \forall \zeta\in\R^d, \quad \|\widehat \cM_u(\zeta)\|_\infty \leq Nc\, \rho^u.
\end{equation}
We set, with $t_0$ given in \textbf{(AC\ref{AS2})},
$$K :=\max(2,Nc)
\quad \text{and} \quad \Gamma(\zeta) = \max\big(2\rho^{t_0},\sup_{w\in[t_0,2t_0)}\|\widehat \cG_w(\zeta)\|_{\infty}\big).$$ 
Let us prove by induction that the following inequality holds for any $n\in\N^*$ and $v\in[0,t_0)$: 
\begin{equation} \label{ineg-mat-hat-g}
\forall \zeta\in\R^d,\quad \| \widehat \cG_{nt_0+v}(\zeta) \|_{\infty} \leq \big[1+(1+K)(n-1)\big]\, \Gamma(\zeta)\, \big( \Gamma(\zeta) + K \rho^{t_0}\big)^{n-1}. 
\end{equation}
First, (\ref{ineg-mat-hat-g}) is obvious for $n:=1$. Second assume that (\ref{ineg-mat-hat-g}) holds for some $n\in\N^*$. Then, using (\ref{hat-G-presque-semi}) with $t:=t_0+v$ and $s:=nt_0$, it  follows from (\ref{masse-id-bis}) and the definition of $\Gamma(\cdot)$ that  
\begin{eqnarray*}
\| \widehat \cG_{(n+1)t_0+v}(\zeta) \|_{\infty} &=& \| \widehat \cG_{t_0+v+nt_0}(\zeta) \|_{\infty} \\
&\leq& \Gamma(\zeta)\, \|\widehat \cG_{nt_0}(\zeta)\|_{\infty}  +  \Gamma(\zeta)\, K\, \rho^{nt_0} +  K\, \rho^{t_0} \|\widehat \cG_{nt_0}(\zeta)\|_{\infty} + 2\, K^2\rho^{t_0+nt_0} \\
&\leq&  \big(\Gamma(\zeta) + K\rho^{t_0}\big)\|\widehat \cG_{nt_0}(\zeta)\|_{\infty} +  \Gamma(\zeta)\, K \rho^{nt_0} + 2\, K^2\rho^{t_0+nt_0} \\
\text{(by induction)} &\leq& \big[1+(1+K)(n-1)\big]\, \Gamma(\zeta)\, \big(\Gamma(\zeta) + K \rho^{t_0}\big)^n  + \Gamma(\zeta)\, K \rho^{nt_0} + 2\, K^2\rho^{t_0+nt_0}.
\end{eqnarray*}
Next, using $K\rho^{nt_0} \leq (K\rho^{t_0})^n \leq (\Gamma(\zeta) + K\rho^{t_0})^n$ and $2\rho^{t_0} \leq \Gamma(\zeta)$, we obtain 
\begin{eqnarray*}
\Gamma(\zeta)\, K \rho^{nt_0} + 2\, K^2\rho^{t_0+nt_0} &\leq& \Gamma(\zeta)\, (\Gamma(\zeta) + K\rho^{t_0})^n + 2K\rho^{t_0}\, (\Gamma(\zeta) + K\rho^{t_0})^n \\
&\leq& (1+K)\, \Gamma(\zeta)\, (\Gamma(\zeta) + K\rho^{t_0})^n. 
\end{eqnarray*}
Hence 
$$\| \widehat \cG_{(n+1)t_0+v}(\zeta) \|_{\infty} \leq \big[1+(1+K)n\big]\, \Gamma(\zeta)\, \big(\Gamma(\zeta) + K \rho^{t_0}\big)^n.$$
The induction is complete and proves (\ref{ineg-mat-hat-g}). 

Now let $t\in[t_0,+\infty)$. By writing $t=nt_0+v$ with $n\in\N^*$ and $v\in[0,t_0)$, Lemma~\ref{lem-hat-g} follows from (\ref{ineg-mat-hat-g}) and Condition~\textbf{(AC\ref{AS2})}. Indeed (\ref{equiv-norm}) gives  
$\Gamma(\zeta)\leq \max\big(2\rho^{t_0},N\Gamma_{t_0}(\zeta)\big)$ 
with  $\Gamma_{t_0}(\cdot)$ defined in (\ref{fourier-id}). From \textbf{(AC\ref{AS2})} we have $K\rho^{t_0}\leq 1/4$, and from (\ref{fourier-id}) and $K\ge 2$ we deduce that  there exists $A>0$ such that $\sup_{\|\zeta\|\geq A} \Gamma(\zeta) \leq 1/4$. Using $n=(t-v)/t_0$ with $v\in[0,t_0)$ and (\ref{equiv-norm}), we derive the expected inequality in Lemma~\ref{lem-hat-g}. 
\end{proof}
%
 \subsection{Proof of Theorem~\ref{main-id}} \label{Main}

 The density $f_{k,t}$ of the a.c.~part of the probability distribution of $t^{-1/2}Y_t$ under $\P_k$ is given by  (see (\ref{loi-Yt}))
$$\forall y \in \cE_t := t^{-1/2} \cD_t, \qquad f_{k,t}(y) = t^{d/2}\, g_{k,t}(t^{1/2} y ).$$
Denote the closure and the frontier of $\cE_t$ by $\overline{\cE}_t$ and $\partial\cE_t$ respectively. 
The following properties of $f_{k,t}$ are easily deduced from \textbf{(AC\ref{AS0})}. 
\begin{lem} \label{Bornes_F} 
Condition~\emph{\textbf{(AC\ref{AS0})}} is assumed to hold. Then, for each $k=1,\ldots,N$ and for all $t>0$, the function $f_{k,t}$ is continuous on  $\overline{\cE}_t$ and differentiable on $\cE_t$. In addition, we have   
\begin{gather*}
\sup_{y\in\overline{\cal E}_t} |f_{k,t}(y)| = O(t^{d/2})  \qquad 
\sup_{y\in \partial {\cal E}_t} |f_{k,t}(y)| = O(t^{-1/2}) \qquad
\max_{1\leq j \leq d}\, 
\sup_{y\in{\cal E}_t}\big|\frac{\partial f_{k,t}}{\partial y_j}(y)\big| = O(t^{(d+1)/2}). 
\end{gather*}
\end{lem}

For any $t>0$, set $d(y,\partial\cE_t) := \inf\{d(y,z),\, z\in\partial\cE_t\}$ and  define 
\begin{equation} \label{def-D't}
\cE'_t := \big\{y\in\cE_t,\ d(y,\partial\cE_t)> t^{-(d/2+1)}\big\},
\end{equation}
 Theorem~\ref{main-id} follows from the two following propositions. 
\begin{pro} \label{pro-second-id} Under Condition~\emph{\textbf{(AC\ref{AS0})}}, there is a positive constant $C$ such that 
\begin{eqnarray*}
\forall t\in\T,\ \forall y\in \R^d\setminus{\cal E}'_t, \quad \big|f_{k,t}(y) - \eta_{I}(y)\big| \leq C\big( t^{-1/2} + \sup_{y\notin {\cal E}_t} \eta_{I}(y) \big).
\end{eqnarray*}
\end{pro} 
\begin{proof}{}
First let $y\in\cE_t\setminus{\cal E}'_t$. By definition of $\cE'(t)$, there exists some $y_t\in\partial\cE_t$ such that $\|y-y_t\| \leq 2\, t^{-(d/2+1)}$. From Taylor's inequality and Lemma~\ref{Bornes_F} it follows that 
\begin{eqnarray*}
\big|f_{k,t}(y) - \eta_{I}(y)\big|	 & \le & \big|f_{k,t}(y) - f_{k,t}(y_t)\big| + 
	\big|f_{k,t}(y_t) - \eta_{I}(y_t)\big| + \big|\eta_{I}(y_t) - \eta_{I}(y) \big|  \\
	&\leq& 2\,  t^{-(d/2+1)} O(t^{(d+1)/2})  + O(t^{-1/2}) + \sup\{\eta_{I}(x): x\in\partial{\cal E}_t\} +  2\, t^{-(d/2+1)} \\
	&\leq& O(t^{-1/2}) + \sup\{\eta_{I}(x): x\in\partial{\cal E}_t\}. 
\end{eqnarray*}
\vspace*{-5mm}

\noindent Second, let $y\in\overline{\cE}_t\setminus{\cal E}_t = \partial{\cal E}_t$. Then $|f_{k,t}(y) - \eta_{I}(y)| \leq O(t^{-1/2}) + \sup\{\eta_{I}(x): x\in\partial{\cal E}_t\}$  
from Lemma~\ref{Bornes_F}. Third, let $y\in\R^d\setminus \overline{\cE}_t$. Then $|f_{k,t}(y) - \eta_{I}(y)|=|\eta_{I}(y)| \leq \sup\{\eta_{I}(x): x\notin {\cal E}_t\}$. The proof of  Proposition~\ref{pro-second-id} is complete. 
\end{proof}
\begin{pro} \label{pro-main-id}
Assume that Conditions~\emph{\IP}, $\M{3}$, 
and \emph{\textbf{(AC\ref{AS2})-(AC\ref{AS0})}} hold true. Then there exists a positive constant $C$ such that
$$\forall t\in\T,\ \forall y\in {\cal E}'_t,\quad \big|f_{k,t}(y) - \eta_{I}(y)\big| \leq C \, t^{-1/2}.
\vspace*{-1mm}
$$
\end{pro}
Note that the inequality above is valid for $k\in\X$ and $t\in[0,t_0]$ where $t_0$ is given in \textbf{(AC\ref{AS2})} since $\eta_{I}$ is bounded on $\R^d$ and  $f_{k,t}$ is uniformly bounded on $\R^d$ with respect to $k$ and $t\in[0,t_0]$  from Lemma~\ref{Bornes_F}. 

\noindent\begin{proof}{}
For any $p\in\N^*$, let $\eta_{p,I}$ be the function defined by $\forall y\in\R^d,\ \eta_{p,I}(y) := p^d\, \eta_{I}(py)$. The usual convolution product on $\R^d$ is denoted by the symbol $\star$. Write for all $p\in\N^*$: 
\begin{eqnarray}
\lefteqn{\big|f_{k,t}(y) - \eta_{I}(y)\big|} \nonumber\\
&  \leq & \big|f_{k,t}(y) - (\eta_{p,I} \star f_{k,t})(y)\big| + \big|(\eta_{p,I} \star f_{k,t})(y) - (\eta_{p,I} \star \eta_{I})(y)\big| + \big|(\eta_{p,I} \star \eta_{I})(y) - \eta_{I}(y)\big| \nonumber \\
& & := K_{1,p}(y) + K_{2,p}(y) + K_{3,p}(y). \label{recap-K}
\end{eqnarray}
The conclusion of Proposition~\ref{pro-main-id} follows from the two next lemmas. 
\end{proof}

For $t>t_0$ we define the following positive integer $p_t := \lfloor t^{d+3/2}\rfloor + 1$.
\begin{lem} \label{lem-K1-K2}
There exists $D>0$ such that  
$$\forall t\in [t_0,+\infty),\ \forall p\geq p_t,\ \forall y\in {\cal E}'_t,\quad K_{1,p}(y) + K_{3,p}(y) \leq D \, t^{-1/2}.$$
\end{lem}
\begin{lem} \label{K2-pt}
There exists $E>0$ such that  
$$\forall t\in [t_0,+\infty),\ \forall y\in \R^d,\quad K_{2,p_t}(y) 
\leq E \, t^{-1/2}.$$
\end{lem}
\begin{proof}{ of Lemma~\emph{\ref{lem-K1-K2}}}
Note that $\int_{\R}\eta_{p,I}(u)du=1$.  Let $y\in\cE'_t$.   Using the definition of $\cE'_t$ (see (\ref{def-D't})) and the fact that $\cE_t$ is open, one can prove that, for all $u\in\R^d$ 
such that $\|u\|\leq t^{-(d/2+1)}$, we have $y+u\in\overline{\cE}_t$ (prove that $\sup\{a\in[0,1] : y+au \in\cE_t\} = 1$). By applying Taylor's inequality to $f_{k,t}$ and using Lemma~\ref{Bornes_F}, we obtain 
\begin{eqnarray*}
\big|f_{k,t}(y) - (\eta_{p,I} \star f_{k,t})(y)\big| &\leq& \int_{\|u\| \leq t^{-(d/2+1)}}   \eta_{p,I}(u)\, \big| f_{k,t}(y) - f_{k,t}(y-u)\big|\, du  \\
&\ & \qquad \qquad + \ \  \int_{\|u\|> t^{-(d/2+1)}} 
 \eta_{p,I}(u)\, \big|f_{k,t}(y) - f_{k,t}(y-u)\big|\, du \\
&\leq&  O\big(t^{-1/2}\big) +  O(t^{d/2}) \int_{\|v\| > p\,  t^{-(d/2+1)} } \eta_{I}(v) \, dv.
\end{eqnarray*}
Next, if $p\geq t^{d+3/2}$,  
then Markov's inequality gives 
$$\int_{\|v\| > p\,  t^{-(d/2+1)}} \eta_{I}(v) \, dv \leq (2\pi)^{-d/2} \int_{\|v\| > \, t^{(d+1)/2}}  e^{-\|v\|^2/2} \, dv =  O\big(t^{-(d+1)/2}\big).$$
This gives the claimed conclusion for $K_{1,p}(y)$. Similarly we can prove that $K_{3,p}(y) = O(t^{-1/2})$ using the fact that $\eta_{I}(\cdot)$ and its differential are bounded on $\R^d$. 
\end{proof}
\begin{proof}{ of Lemma~\emph{\ref{K2-pt}}}
Since $f_{k,t}(y)= t^{d/2} g_{k,t}(t^{1/2}y)$, we have $\widehat f_{k,t}(\zeta) = \widehat g_{k,t}(t^{-1/2}\zeta)$. For any $t>t_0$ 
and $\zeta\in\R^d$, we set 
$$\Delta_{k,t}(\zeta) := \big|\widehat g_{k,t}(t^{-1/2}\zeta) - e^{-\|\zeta\|^2/2}\big|.$$ 
Let $0<\delta<A$ (fixed) be given by Lemmas~\ref{lem-dec-vp} and \ref{lem-hat-g}. 
From $0<\widehat\eta_{p,I}(\zeta) = \widehat\eta_{I}(p^{-1}\zeta)\leq 1$ and the inverse Fourier formula, the following inequality holds for all $p\in\N^*$ and  $y\in\R^d$:  
\begin{eqnarray*}
(2\pi)^d\, K_{2,p}(y) &\leq& \int_{\R}\Delta_{k,t}(\zeta)\, \widehat\eta_{I}(p^{-1}\zeta)\,  d\zeta \\
& \leq & \int_{\|\zeta\|\leq\delta\sqrt t}\Delta_{k,t}(\zeta)\, d\zeta + \int_{\delta\sqrt t\leq\|\zeta\|\leq A\sqrt t}\Delta_{k,t}(\zeta)\, d\zeta + \int_{\|\zeta\|\geq A\sqrt t} \Delta_{k,t}(\zeta)\, \widehat\eta_{I}(p^{-1}\zeta)\, d\zeta \\[0.15cm]
& & \  := J_1(t) + J_2(t) + J_{3,p}(t).
\end{eqnarray*} 
From (\ref{loi-Yt}) we obtain $\phi_{k,t}(\zeta) = \E_{k}\big[e^{i \zeta Y_t} \, \big] = \widehat g_{k,t}(\zeta) + \widehat \mu_{k,t}(\zeta)$. 
Using this equality and (\ref{masse-id}), we obtain for all $t>t_0$:
\begin{eqnarray*}
J_1(t) &\leq& \int_{\|\zeta\|\leq\delta\sqrt t} \big|\phi_{k,t}(t^{-1/2}\zeta) - e^{-\|\zeta\|^2/2}\big|\, d\zeta +   c\rho^t O(t^{d/2}) \  := \ I_1(t) + O\big( t^{-1/2}\big). \\
J_2(t) &\leq& \int_{\delta\sqrt t\leq\|\zeta\|\leq A\sqrt t} \big|\phi_{k,t}(t^{-1/2}\zeta) - e^{-\|\zeta\|^2/2}\big|\, d\zeta +  c\rho^t O(t^{d/2}) \ := \ I_2(t) + O\big( t^{-1/2}\big). 
\end{eqnarray*}
Now  we obtain from Lemma~\ref{lem-dec-vp}  and Lemma~\ref{lem-non-ari} respectively
\begin{eqnarray*}
I_1(t) &\leq& \int_{\|\zeta\|\leq\delta\sqrt t} \bigg|\lambda(t^{-1/2}\zeta)^{\lfloor t \rfloor}\, L_{k,t}(t^{-1/2}\zeta)  + R_{k,t}(t^{-1/2}\zeta)  - e^{-\|\zeta\|^2/2} \bigg| \, d\zeta \\
&\leq& \int_{\|\zeta\|\leq\delta\sqrt t} \bigg(|L_{k,t}(t^{-1/2}\zeta)|\, \big|\lambda(t^{-1/2}\zeta)^{\lfloor t \rfloor} - e^{-\|\zeta\|^2/2}\big| \\
& & \qquad \qquad \qquad \qquad + \qquad e^{-\|\zeta\|^2/2}\big|L_{k,t}(t^{-1/2}\zeta)-1\big| + |R_{k,t}(t^{-1/2}\zeta)|
\bigg) \, d\zeta \\
&\leq& t^{-1/2} C \bigg[(1+\delta C)\int_{\R^d} (1+\|\zeta\|^3)e^{-\|\zeta\|^2/8} \, d\zeta 
   +  \int_{\R^d} \|\zeta\|\, e^{-\|\zeta\|^2/2}\, d\zeta \bigg] + C\, r^{\lfloor t \rfloor} \int_{\|\zeta\|\leq\delta\sqrt t} \, d\zeta \\
&\leq& O\big(t^{-1/2}\big) + O\big(t^{d/2}r^{\lfloor t \rfloor} \big);\\
I_2(t) &\leq& \int_{\delta\sqrt t\leq\|\zeta\|\leq A\sqrt t} \big|\phi_{k,t}(t^{-1/2}\zeta)\big| \, d\zeta + \int_{\delta\sqrt t\leq\|\zeta\|\leq A\sqrt t} e^{-\|\zeta\|^2/2}\, d\zeta \\
&\leq& O\big(t^{d/2}\tau^{\lfloor t \rfloor}\big) +  \int_{\|\zeta\|\geq \delta\sqrt t} e^{-\|\zeta\|^2/2}\, d\zeta. 
\end{eqnarray*}
Thus, we have proved that there exists a constant $e_1>0$ such that 
\begin{equation} \label{J1-J2}
\forall t>t_0,\quad J_1(t) + J_2(t) \leq e_1\, t^{-1/2}.
\end{equation}
It remains to prove that: $\forall t>t_0,\ J_{3,p_t}(t) = O(t^{-1/2})$. 
We have for any $p\geq1$ 
\begin{eqnarray*}
J_{3,p}(t) &=& \int_{\|\zeta\|\geq A\sqrt t} \big|\widehat g_{k,t}(t^{-1/2}\zeta) - e^{-\|\zeta\|^2/2}\big|\, \widehat\eta_{I}(\zeta/p)\, d\zeta \\
&\leq& \int_{\|\zeta\|\geq A\sqrt t} \big|\widehat g_{k,t}(t^{-1/2}\zeta)\big|\, \, \widehat\eta_{I}(p^{-1}\zeta)\, d\zeta +  \int_{\|\zeta\|\geq A\sqrt t} e^{-\|\zeta\|^2/2}\, d\zeta
\end{eqnarray*}
The last integral is $O(t^{-1/2})$. Next, picking $p=p_t$,  we deduce from Lemma~\ref{lem-hat-g} that: $\forall t>t_0$
\begin{eqnarray*}
 \int_{\|\zeta\|\geq A\sqrt t} \big|\widehat g_{k,t}(t^{-1/2}\zeta)\big|\, \widehat\eta_{I}({p_t}^{-1}\zeta)\, d\zeta 
&\leq& NC\, \frac{t}{2^{t/t_0}}\int_{\R^d}\widehat\eta_{I}({p_t}^{-1}\zeta) \, d\zeta 
\leq NC\, \frac{t\, {p_t}^d}{2^{t/t_0}} \int_{\R^d}\widehat\eta_{I}(\zeta) \, d\zeta.
\end{eqnarray*}
Thus there exists a constant $e_2>0$ such that: $\forall t\geq t_0,\ J_{3,p_t}(t) \leq e_2\, t^{-1/2}$. \end{proof}

\begin{rem} \label{rque-unif}
The uniform version of Theorem~\ref{main-id} in Remark~\ref{main-unif} holds because the conclusions of Lemmas~\ref{lem-dec-vp}, \ref{lem-non-ari} and \ref{lem-hat-g} are valid uniformly in  $P\in \cP_0$ under \emph{\textbf{(U\ref{U1})-(U\ref{U3})}}. This is obvious for  Lemma~\ref{lem-hat-g}. The uniformity in  $P\in \cP_0$ for Lemmas~\ref{lem-dec-vp} and  \ref{lem-non-ari} is derived from the fact that the map $(P,\zeta)\mapsto \widehat{\cY_1}(\zeta)$ is continuous  from $\cP_0\times\R^d$ into $\cM_N(\C)$ under  \emph{\textbf{(U\ref{U1})}} when the moment condition $\M{3}$ holds uniformly in $P\in \cP_0$. See Section~6 in [HL] for details.
\end{rem}

\end{document}